\documentclass[a4paper,12pt]{elsarticle}
\usepackage[utf8]{inputenc}
\usepackage{amssymb}
\usepackage{amsfonts}
\usepackage{amsmath,amsthm}

\usepackage{enumerate}

\newtheorem{theorem}{Theorem}

\newtheorem{conjecture}[theorem]{Conjecture}

\newtheorem{corollary}[theorem]{Corollary}
\newtheorem{lemma}[theorem]{Lemma}

\newcommand{\probP}{{\mathbb{P}}}

\begin{document}

\title{Distinguishing adjacent vertices by ordering edges}

\author[AGH]{Aleksandra Gorzkowska} 
\ead{agorzkow@agh.edu.pl}

\author[AGH]{Jakub Kwa\'sny\corref{cor1}} 
\ead{jkwasny@agh.edu.pl}

\address[AGH]{AGH University of Krakow, al. A. Mickiewicza 30, 30-059 Krakow, Poland}
\cortext[cor1]{Corresponding author}

\begin{abstract}
    The 1-2-3 Conjecture states that for every graph without isolated edges, there exists an edge-weighting from $\{1,2,3\}$ such that adjacent vertices receive distinct sums of weights on their incident edges. In the sequence variant, adjacent vertices are to be distinguished by the sequences of weights on their  incident edges. In this paper, we investigate whether, for every graph without isolated edges and for a fixed proper edge colouring (where colours are interpreted as weights), there exists a global total order of the edges such that the resulting sequences of incident weights distinguish adjacent vertices. For connected graphs, we prove that such an order exists whenever there exist two adjacent vertices that have distinct sets of incident weights. This yields a positive answer for every proper edge colouring of a connected non-regular graph or a connected graph of class two. Moreover, a probabilistic argument gives the same conclusion for connected regular graphs of degree at least six. 
\end{abstract}

%\noindent\emph{AMS Subject Classification:} 05C15, 05C20, 05C25, 05C05

%\begin{keyword}
%distinguishing index  \sep
%automorphism group \sep
%symmetry breaking \sep
%graph coloring \sep
%oriented graph
%\end{keyword}

\maketitle

\section{Introduction}

This paper concerns a sequence variant of the 1-2-3 Conjecture, an influential and widely studied problem in graph weighting that remained open for almost twenty years. Introduced in 2004 by Karoński, Łuczak and Thomason \cite{KaronskiLuczakThomason}, the conjecture states that every connected graph distinct from $K_2$ admits an edge weighting with weights in $\{1,2,3\}$ such that adjacent vertices receive distinct sums of incident weights. It was proved by Keusch \cite{Keusch}. Alongside efforts to resolve the conjecture, a rich collection of related problems emerged, varying either the weights used or the rule that assigns vertex colours from incident edge weights. These include list, set and sequence variants; see Seamone \cite{SeamoneSurvey} for a survey of their early development.

The sequence variant was proposed by Seamone and Stevens \cite{SeamoneStevens}. Let $G$ be a connected graph distinct from $K_2$, and let $\prec$ be a total order on $E(G)$. Given an edge weighting $w: E(G) \rightarrow \mathcal C$, we define the \emph{vertex colouring induced by $w$ and $\prec$} by
\[
    f_{\prec}(v)=\bigl(w(e_1),\dots,w(e_{d(v)})\bigr),
\]
where $e_1\prec\dots\prec e_{d(v)}$ are all the edges incident with $v$. When the order is clear from the context, we write $f(v)$ instead of $f_{\prec}(v)$. Seamone and Stevens studied the choice of edge weights for a fixed total order, with the aim of making the induced vertex colouring proper.

Unlike sums, sequences depend on the order of the incident edges. Thus, adjacent vertices with the same multiset of incident weights may still receive different sequences. Sequence variants have also been studied in the setting of irregularity strength; see, for example, the work on hypercubes in \cite{Hypercubes}.

In this paper, we reverse the roles of the weights and the order. We fix a proper edge colouring $w$ and ask whether there is a total order $\prec$ on $E(G)$ for which $f_{\prec}$ is a proper vertex colouring. We call $w$ \emph{sequentially orderable} if such an order exists.

The restriction to proper edge colourings gives the edge order the greatest possible local freedom. An arbitrary weighting may offer no such freedom: if two adjacent vertices have the same degree and all edges incident with either vertex have weight $a$, both vertices receive the same constant sequence under every order. In contrast, pairwise distinct incident weights at a vertex of degree $d$ give $d!$ distinct sequences as its incident edges are reordered, the maximum possible number. The challenge is to choose a single global edge order that distinguishes all adjacent vertices simultaneously.

Our main structural result shows that a proper edge colouring of a connected graph is sequentially orderable whenever there exist two adjacent vertices that have distinct sets of incident weights. In particular, every proper edge colouring of a connected non-regular graph or a connected graph of class two is sequentially orderable. A probabilistic argument establishes the same conclusion for connected regular graphs of degree at least six.

Trivially, the colouring of $K_2$ and proper 2-edge-colourings of even cycles are not sequentially orderable. Our results leave open only proper $d$-edge-colourings of 2-connected $d$-regular graphs with $d\in\{3,4,5\}$. We conjecture that every proper edge colouring of a connected graph other than $K_2$ or an even cycle is sequentially orderable.

\section{Preliminaries}

For graph theory notation and terminology, we follow \cite{Diestel}. Graphs are finite and simple unless stated otherwise. A graph $G$ is of \emph{class one} if $\chi'(G)=\Delta(G)$ and of \emph{class two} if $\chi'(G)=\Delta(G)+1$.

The following lemma gives a greedy construction that distinguishes all adjacent pairs not involving a chosen vertex. In this proof, we use an operation of \emph{inserting} a new edge into an order, which means choosing its position among the already ordered edges while preserving their relative order.

\begin{lemma} Let $G$ be a connected loopless multigraph with a proper edge colouring $w$, and let $v\in V(G)$. There is a total order $\prec$ on $E(G)$ such that $f_{\prec}(x)\neq f_{\prec}(y)$ for every edge $xy\in E(G-v)$. \label{lem: single_bad_vertex}
\end{lemma}
\begin{proof}
    Choose a spanning tree of $G$ rooted at $v$. Take an ordering $v_1,\dots,v_n=v$ in which every non-root vertex precedes its parent. In particular, each $v_i$ with $i<n$ has a neighbour with larger index. We process $v_1,\dots,v_{n-1}$ in this order, inserting their incident edges into a growing total order.

    Suppose that $v_1,\dots,v_{i-1}$ have been processed and all their incident edges have been ordered. At least one edge at $v_i$ is still unordered, since its parent has not yet been processed. Insert all but one of the unordered incident edges arbitrarily, leaving an edge $e$. There are exactly $d(v_i)$ positions for $e$ relative to the other incident edges. These yield pairwise distinct sequences $f(v_i)$ because $w$ is proper. On the other hand, at most $d(v_i)-1$ neighbours of $v_i$ have been processed, so we can choose a position avoiding all their sequences. Since all edges at a processed vertex are already ordered, later insertions leave its sequence unchanged.
\end{proof}

The proof of Lemma \ref{lem: single_bad_vertex} uses the property that $w$ is a proper colouring only at the vertices processed before $v$. Thus, the lemma also applies when this condition fails at $v$, provided it holds at every other vertex. 

\begin{corollary} Let $G$ be a connected graph with a vertex $v$ such that $d(v)\neq d(u)$ for every $u\in N(v)$. Then every proper edge colouring of $G$ is sequentially orderable. \label{cor: unique_degree}
\end{corollary}
\begin{proof}
    Apply Lemma \ref{lem: single_bad_vertex} with exceptional vertex $v$. The sequence at $v$ differs in length from that of its every neighbour, so the induced vertex colouring by sequences is proper.
\end{proof}

We next use Lemma \ref{lem: single_bad_vertex} to prove results for graphs with a cut-edge or a cut-vertex. We order the components separately and then merge their orders.

\begin{lemma} Let $G\neq K_2$ be a connected graph with a cut-edge. Then every proper edge colouring of $G$ is sequentially orderable. \label{lem: cut-edge}
\end{lemma}
\begin{proof}
    Fix a proper edge colouring $w$ and a cut-edge $xy$. Let $G_1,G_2$ be the components of $G-xy$ containing $x,y$, respectively. If either component is a single vertex, exactly one of $x,y$ has degree one, since $G\neq K_2$, and Corollary \ref{cor: unique_degree} applies.

    Otherwise, let $H_i$ be a graph obtained from $G_i$, for $i=1,2$, by adjoining $xy$ and its other endpoint to it, together with the colouring of edges inherited from $G$. The added endpoint has degree one, and its neighbour has degree at least two. Therefore, by Corollary \ref{cor: unique_degree} there exists an order $\prec_i$ inducing a proper vertex colouring $f_i$ on each $H_i$.

    Since $E(H_1)\cap E(H_2)=\{xy\}$, the orders admit a common extension to $E(G)$. Every vertex of $G_i$ has all its incident edges in $H_i$, so this extension preserves $f_i$ on $V(G_i)$. Thus, only the sequences of the vertices $x$ and $y$ can possibly be equal. If $f_1(x)=f_2(y)$, we reverse $\prec_2$ before merging the orders. This reverses every sequence on $H_2$, preserving the fact that the sequences at adjacent vertices in that graph are distinct. It also changes the sequence at $y$, which consists of at least two pairwise distinct weights, while leaving the sequence at $x$ unchanged. Therefore, a common extension using the reversed order distinguishes $x$ and $y$ as well.
\end{proof}

\begin{lemma} Let $G$ be a connected graph with a cut-vertex. Then every proper edge colouring of $G$ is sequentially orderable. \label{lem: cut-vertex}
\end{lemma}
\begin{proof}
    Fix a proper edge colouring $w$ and a cut-vertex $v$. Let $G_1$ be a component of $G-v$ and $G_2$ the union of the remaining components of $G-v$. Put $k=|N(v)\cap V(G_1)|$ and $l=|N(v)\cap V(G_2)|$. If $k=1$ or $l=1$, the unique edge from $v$ to $G_1$ or $G_2$ is a cut-edge, so by Lemma \ref{lem: cut-edge} the result holds. Assume $k,l\geq 2$.

    Let $H_i=G[V(G_i)\cup\{v\}]$. Each $H_i$ is connected, so by Lemma \ref{lem: single_bad_vertex} there exists an order $\prec_i$ whose induced vertex colouring $f_i$ distinguishes every pair of adjacent vertices in $H_i-v$. Every vertex of $G_i$ has all its incident edges in $H_i$. Hence a common extension of $\prec_1$ and $\prec_2$ preserves its sequence, and only the sequence at $v$ depends on how the orders are merged.

    The restrictions of $\prec_1$ and $\prec_2$ to the edges at $v$ have $\binom{k+l}{k}$ interleavings. Each can be extended to a total order on $E(G)$ preserving both $\prec_1$ and $\prec_2$. Since the weights at $v$ are pairwise distinct, these interleavings yield distinct sequences at $v$. There are at most $d(v)=k+l$ neighbour sequences to avoid, whereas
    \[
        \binom{k+l}{k}\geq\binom{k+l}{2}>k+l
        \qquad\text{for }k,l\geq2.
    \]
    Choose an interleaving avoiding all these sequences and extend it to a total order of $E(G)$. The induced vertex colouring is then proper.
\end{proof}

\section{Distinct palettes}

For a proper edge colouring $w$, the \emph{palette} at a vertex $v$ is the set $S(v)=\{w(vu):u\in N(v)\}$. An edge $xy$ is \emph{safe} if $S(x)\neq S(y)$. Its endpoints then receive distinct sequences under every edge order.

For a connected graph $G$, Lemma \ref{lem: single_bad_vertex} shows that $w$ is sequentially orderable if all edges at some vertex are safe. The next theorem replaces this condition by the existence of a single safe edge.

\begin{theorem} Let $G$ be a connected graph with a proper edge colouring $w$ such that there exist two adjacent vertices with different palettes. Then $w$ is sequentially orderable. \label{thm: safe-edge}
\end{theorem}
\begin{proof}
    Choose a vertex $v$, let $V_1$ consist of all vertices with palette $S(v)$, and put $V_2=V(G)\setminus V_1$. Both sets are nonempty, since some adjacent vertices have distinct palettes. Let $M$ be the set of edges between $V_1$ and $V_2$. Connectivity ensures that $M$ is nonempty, and every edge of $M$ is safe.

    Fix a total order $\prec_M$ on $M$. For each component $Q$ of $G-M$, let $M_Q$ be the set of edges of $M$ incident with $V(Q)$. We will order $E(Q)\cup M_Q$ so that adjacent vertices in $Q$ receive distinct sequences and the restriction to $M_Q$ agrees with $\prec_M$. All sequences considered in this construction include the weights on edges of $M_Q$, so they are the full sequences of vertices of $Q$ in $G$.

    Since $G$ is connected, $M_Q$ is nonempty. Let $e_Q$ be its first edge under $\prec_M$, and let $r$ be the endpoint of $e_Q$ in $Q$. Order the edges of $M_Q\setminus\{e_Q\}$ according to $\prec_M$ and reserve them as a final block. All subsequent insertions will be made before this block. If the block is empty, there is no restriction on the insertion positions.

    Root a spanning tree of $Q$ at $r$. As in Lemma \ref{lem: single_bad_vertex}, process all vertices other than $r$ before their parents. At a vertex $u\neq r$, the edge to its parent is still unordered. Insert all but one of the unordered edges of $Q$ incident with $u$ before the final block, leaving an edge $e$. There are $d_Q(u)$ positions for $e$ relative to the other incident edges of $Q$, all available before the block. These give distinct full sequences at $u$, since $w$ is proper and the incident weights from $M_Q$ form a fixed suffix. At most $d_Q(u)-1$ neighbours in $Q$ have been processed, so a position avoiding all their sequences exists. Neighbours outside $Q$ need not be considered, since the edges joining them to $u$ are safe. Later insertions leave the sequence at $u$ unchanged.

    Once all vertices other than $r$ have been processed, every edge of $Q$ is ordered. Insert $e_Q$ before the final block. Its $d_Q(r)+1$ positions relative to the edges of $Q$ incident with $r$ give distinct full sequences, while only $d_Q(r)$ neighbour sequences in $Q$ must be avoided. Thus a suitable position exists. This also covers the case when $Q$ is a single vertex. The resulting order distinguishes all adjacent vertices in $Q$ and agrees with $\prec_M$ on $M_Q$, because $e_Q$ precedes the final block.

    Finally, the orders constructed for the components admit a common extension together with $\prec_M$. Take such a common extension to $E(G)$. Every vertex of $Q$ has all its incident edges in $E(Q)\cup M_Q$, so its sequence is preserved. All edges outside $M$ therefore have distinct endpoint sequences, and the same holds for the safe edges of $M$.
\end{proof}

The following two statements are immediate consequences of Theorem \ref{thm: safe-edge}.

\begin{corollary} Let $G$ be a connected graph that is not a regular graph. Then every proper edge colouring of $G$ is sequentially orderable.
\end{corollary}

\begin{corollary} Let $G$ be a connected graph of class two. Then every proper edge colouring of $G$ is sequentially orderable.
\end{corollary}

\section{Regular graphs}

We first consider the two elementary obstructions. No edge weighting of $K_2$, the only connected 1-regular graph, is sequentially orderable, since its endpoints receive the same one-term sequence.

Likewise, no proper 2-edge-colouring of an even cycle is sequentially orderable. To see this, take the first edge in any total order and denote its weight by $a$. At each endpoint, the other incident edge has the second weight $b$ and occurs later in the order. Both endpoints therefore receive the sequence $(a,b)$.

Every proper edge colouring of a cycle using at least three colours has a safe edge: otherwise, all vertices would have the same two-element palette. Theorem \ref{thm: safe-edge} therefore applies to every such colouring, including every proper edge colouring of an odd cycle.

For regular graphs of degree at least six, we use a random edge order. The proof relies on the following symmetric form of the Local Lemma.

\begin{theorem}[The Local Lemma \cite{AlonSpencer}]
\label{LLL-symmetric}
Let $A_1,\dots,A_n$ be events in a common probability space, with $\probP(A_i)\leq p$ for every $i$.
Suppose each $A_i$ is independent of a family containing all but at most $D$ of the other events. If
\[
    ep(D+1)\leq1,
\]
then $\probP\bigl(\bigcap_{i=1}^n\overline{A_i}\bigr)>0$.
\end{theorem}

\begin{theorem} Let $G$ be a connected $d$-regular graph with $d\geq 6$. Then every proper edge colouring $w$ of $G$ is sequentially orderable.
\end{theorem}
\begin{proof}
    Assign independent random variables, each uniform on $(0,1)$, to the edges and order the edges by increasing values. Almost surely there are no ties, and the resulting total order $\prec$ is uniform. For each edge $uv$, let $A_{uv}$ be the event that $f_{\prec}(u)=f_{\prec}(v)$. We will bound the probabilities of these events and the number of possible dependencies between them.

    If $S(u)\neq S(v)$, then $\probP(A_{uv})=0$. Suppose now that $S(u)=S(v)$. The sets of edges incident with $u$ and $v$ share only $uv$, so their union has $2d-1$ edges, whose relative orders are equally likely.

    To count the orders for which $A_{uv}$ occurs, fix the position $j+1$ of $w(uv)$ in the common sequence, where $0\leq j\leq d-1$. There are $(d-1)!$ orders of the remaining weights. Each determines the incident edge order at both endpoints. The two lists of $j$ edges preceding $uv$ can be interleaved in $\binom{2j}{j}$ ways, and the two lists following it in $\binom{2d-2-2j}{d-1-j}$ ways. Hence
    \[
        \probP(A_{uv})
        =\frac{(d-1)!}{(2d-1)!}
        \sum_{j=0}^{d-1}\binom{2j}{j}\binom{2d-2-2j}{d-1-j}
        =\frac{(d-1)!\,4^{d-1}}{(2d-1)!}
        =:p_d.
    \]
    The sum is evaluated using identity (3.90) in \cite{Gould}. Thus $\probP(A_{uv})\leq p_d$ for every edge $uv$, with equality when the endpoint palettes coincide.

    Next, $A_{uv}$ depends only on the variables assigned to edges incident with $u$ or $v$, so this event is independent of the family of events $A_{xy}$ with $x,y\notin N(u)\cup N(v)$.
    To bound the number of remaining events, first count the $2(d-1)$ edges other than $uv$ incident with $u$ or $v$. Each of the at most $2(d-1)$ other neighbours of $u$ and $v$ is incident with at most $d-1$ further edges. Thus we may take
    \[
        D=2(d-1)+2(d-1)^2=2d(d-1).
    \]

    Finally, put $q_d=e p_d(2d(d-1)+1)$. The Local Lemma requires $q_d\leq1$. Since $p_6=32/10395$, we have
    \[
        q_6=\frac{1952e}{10395}<1.
    \]
    Moreover,
    \[
        \frac{q_{d+1}}{q_d}
        =\frac{2}{2d+1}\,
        \frac{2d(d+1)+1}{2d(d-1)+1}<1
        \qquad\text{for }d\geq6.
    \]
    Hence $q_d<1$ for every $d\geq6$. By the Local Lemma, with positive probability none of the events $A_{uv}$ occurs, so the induced vertex colouring is proper.
\end{proof}

\section{Concluding remarks}

The only unsolved cases of proper $d$-edge-colourings of 2-connected $d$-regular graphs are $d\in\{3,4,5\}$. We conjecture that these colourings are also sequentially orderable, leaving colourings of $K_2$ and proper 2-edge-colourings of even cycles as the only exceptions. In view of the results above, this is equivalent to the following statement.

\begin{conjecture}
Let $G$ be a connected graph other than $K_2$ or an even cycle. Then every proper edge colouring of $G$ is sequentially orderable.
\end{conjecture}

\section*{Funding}
This research was supported by the AGH University of Krakow under grant no. 16.16.420.054, funded by the Polish Ministry of Science and Higher Education.

\end{document}